\documentclass[11pt]{amsart}
\usepackage{amsmath}
\usepackage{amssymb}
\usepackage{mathtools}      
\usepackage{mathabx} 
\usepackage{bbm}	
\usepackage[normalem]{ulem} 
\usepackage{enumitem}
\usepackage{amsthm}
\usepackage{thmtools}
\usepackage[colorlinks=true,backref=page,bookmarksopen=true]{hyperref} 
\usepackage{caption} 
\usepackage{tikz}
\usetikzlibrary{patterns,decorations.pathreplacing}

\usepackage{xcolor}
\hypersetup{
    colorlinks,
    linkcolor={red!50!black},
    citecolor={blue!50!black},
    urlcolor={blue!80!black}
}

\renewcommand*{\backref}[1]{}

\renewcommand*{\backrefalt}[4]{\ \tiny 
  \ifcase #1 ({\color{red} \bf NOT CITED.})%
  \or    ($\uparrow$#2)%
  \else   ($\uparrow$#2)%
  \fi}

\declaretheorem{theorem}

\declaretheorem[sibling=theorem]{lemma}

\declaretheorem[sibling=theorem, style=remark]{remark}

\hypersetup{bookmarksdepth = 3} 

\setlist[enumerate,1]{label={\upshape(\alph*)},ref=\alph*}

 \newcommand{\Z}{\mathbb{Z}}  \newcommand{\R}{\mathbb{R}} 
\newcommand{\E}{\mathbf{E}}

\newcommand{\st}{\;\mathord{:}\;}

\newcommand{\GL}{\mathrm{GL}}
\DeclareMathOperator{\supp}{supp}

\renewcommand{\epsilon}{\varepsilon}
\renewcommand{\phi}{\varphi}
\renewcommand{\setminus}{\smallsetminus}
\renewcommand{\emptyset}{\varnothing}
\renewcommand{\angle}{\measuredangle}

\DeclareFontFamily{U} {MnSymbolA}{} 
\DeclareFontShape{U}{MnSymbolA}{m}{n}{
   <-6> MnSymbolA5
   <6-7> MnSymbolA6
   <7-8> MnSymbolA7
   <8-9> MnSymbolA8
   <9-10> MnSymbolA9
   <10-12> MnSymbolA10
   <12-> MnSymbolA12}{}
\DeclareFontShape{U}{MnSymbolA}{b}{n}{
   <-6> MnSymbolA-Bold5
   <6-7> MnSymbolA-Bold6
   <7-8> MnSymbolA-Bold7
   <8-9> MnSymbolA-Bold8
   <9-10> MnSymbolA-Bold9
   <10-12> MnSymbolA-Bold10
   <12-> MnSymbolA-Bold12}{}
\DeclareSymbolFont{MnSyA} {U} {MnSymbolA}{m}{n}
\DeclareFontFamily{U} {MnSymbolC}{}
\DeclareFontShape{U}{MnSymbolC}{m}{n}{
  <-6> MnSymbolC5
  <6-7> MnSymbolC6
  <7-8> MnSymbolC7
  <8-9> MnSymbolC8
  <9-10> MnSymbolC9
  <10-12> MnSymbolC10
  <12-> MnSymbolC12}{}
\DeclareFontShape{U}{MnSymbolC}{b}{n}{
  <-6> MnSymbolC-Bold5
  <6-7> MnSymbolC-Bold6
  <7-8> MnSymbolC-Bold7
  <8-9> MnSymbolC-Bold8
  <9-10> MnSymbolC-Bold9
  <10-12> MnSymbolC-Bold10
  <12-> MnSymbolC-Bold12}{}
\DeclareSymbolFont{MnSyC} {U} {MnSymbolC}{m}{n}

\DeclareMathSymbol{\top}{\mathord}{MnSyA}{219} 
\DeclareMathSymbol{\bot}{\mathord}{MnSyA}{217}
\DeclareMathSymbol{\smallplus}{\mathord}{MnSyC}{20} 
\DeclareMathSymbol{\smallminus}{\mathord}{MnSyC}{16} 
\DeclareMathSymbol{\smalltimes}{\mathord}{MnSyC}{21}
\DeclareMathSymbol{\smallpm}{\mathord}{MnSyC}{22} 
\DeclareMathSymbol{\smallmp}{\mathord}{MnSyC}{23}

\begin{document}

\title[The angle between Oseledets spaces]{On the distribution of the angle between Oseledets spaces}
\author{Jairo Bochi}
\address{Department of Mathematics, Penn State University}
\email{\href{mailto:bochi@psu.edu}{bochi@psu.edu}}
\author{Pablo Lessa}
\address{CMAT, Facultad de Ciencias, UdelaR}
\email{\href{mailto:lessa@cmat.edu.uy}{lessa@cmat.edu.uy}}
\date{November 12, 2025}

\begin{abstract}
We study the distribution of the angles between Oseledets subspaces and their log-integrability, focusing on dimension $2$. For random i.i.d.\ products of matrices, we construct examples of probability measures on \(\mathrm{GL}_2(\R)\) with finite first moment where the Oseledets angle is not log-integrable. We also show that for probability measures with finite second moment the angle is always log-integrable. We then consider general measurable \(\mathrm{GL}_2(\R)\)-cocycles over an arbitrary ergodic automorphism of a non-atomic Lebesgue space, proving that no integrability condition on the matrix distribution ensures log-integrability of the angle. In fact, the joint distribution of the Oseledets spaces can be chosen arbitrarily. A similar flexibility result for bounded cocycles holds under an unavoidable technical restriction.  
\end{abstract}

\maketitle

\section{Introduction}

\subsection{The Oseledets splitting}

Let $T$ be an ergodic automorphism of a probability space $(\Omega,\mathcal{S},\mu)$, and let $F$ be a Borel measurable function from $\Omega$ to the matrix group $\mathrm{GL}_d(\R)$. Then there exists a unique matrix-valued function $F^{(n)}(\omega)$, where $\omega \in \Omega$ and $n \in \Z$, such that $F^{(0)}(\omega) = F(\omega)$ and 
\begin{equation}\label{e.cocycle_id}
F^{(m+n)}(\omega) = F^{(m)}(T^n \omega) F^{(n)}(\omega)
\end{equation}
for all $\omega \in \Omega$ and $n \in \Z$. 
Thus, if $n > 0$,
\begin{equation}
F^{(n)}(\omega) = F(T^{n-1} \omega) \cdots F(T\omega) F(\omega) \, .
\end{equation}
Relation \eqref{e.cocycle_id} is called the \emph{cocycle identity}.
We call the pair $(T,F)$ a \emph{measurable linear cocycle}. 
In a more general setting, cocycles form an important class of dynamical systems: see \cite[{\S}1.3.k]{HK}.
Here, we are concerned with the asymptotic behavior of the matrix products $F^{(n)}(\omega)$, for typical points~$\omega$.

Recall that the \emph{singular values} of a matrix $g \in \mathrm{GL}_d(\R)$ are the eigenvalues of the positive symmetric matrix $(g^\top g)^{1/2}$. We denote them, ordered, and repeated according to multiplicity, as $\mathbf{s}_1(g) \ge \cdots \ge \mathbf{s}_d(g)$. In terms of the operator Euclidean norm, $\mathbf{s}_1(g) = \|g\|$ and $\mathbf{s}_d(g) = \|g^{-1}\|^{-1}$. 

We say that the measurable linear cocycle $(T,F)$ is \emph{log-integrable} 
if each of the functions $\log \mathbf{s}_i(F)$ is integrable. Equivalently, 
\begin{equation}\label{e.log-integrable_cocycle}
\int_\Omega \log \max \big( \|F(\omega)\|, \|F(\omega)^{-1}\| \big) \, d\mu(\omega) < \infty 
\end{equation}
(note that the integrand is non-negative).
Under this condition, there exist numbers $\lambda_1 \ge \cdots \ge \lambda_d$, called the \emph{Lyapunov exponents}, such that for each $i \in \{1,\dots, d\}$ and $\mu$-a.e.~$\omega \in \Omega$,
\begin{equation}
\lim_{|n| \to \infty} \frac{1}{|n|} \log \mathbf{s}_i(F^{(n)}(\omega)) = \lambda_i \, ;
\end{equation}
this fact is due to Furstenberg and Kesten \cite{FK} in the i.i.d.\ case and follows from Kingman's subadditive ergodic theorem \cite{kingman} in general.
The Oseledets theorem \cite{oseledets} provides asymptotic information on the norms of the images  $F^{(n)}(\omega) v$ of vectors $v \in \R^d$. It states that for $\mu$-a.e.\ $\omega \in \Omega$, the sets 
\begin{equation}
\mathbf{E}_i(\omega) \coloneqq \left\{ v \in\R^d : \lim_{|n| \to \infty} \frac{1}{n} \log \| F^{(n)}(\omega) v\| = \lambda_i \text{ or } v=0\right\}, \  i=1, \dots ,d,
\end{equation}
are vector subspaces whose union spans $\R^d$ and satisfy $\mathbf{E}_i(\omega) \cap \mathbf{E}_j(\omega) = \{0\}$ whenever $\lambda_i \neq \lambda_j$.
Thus, by removing duplicates, we obtain a splitting of $\R^d$ which depends measurably on $\omega$.
The subspaces $\mathbf{E}_i(\omega)$ are called \emph{Oseledets spaces}.
They are invariant under the action of the cocycle: $F^{(n)}(\omega) \mathbf{E}_i(\omega) = \mathbf{E}_i(T^n \omega)$. 
More information on the Oseledets theorem can be found in the references \cite{arnold,barreira,ledrappier,Viana}. For geometric generalizations of Oseledets theorem, see \cite{filip}.

A positive measurable function $f$ on $\Omega$ is called \emph{tempered} if 
\begin{equation}
\lim_{|n| \to \infty} \frac{1}{n} \log f(T^n \omega) = 0 \quad \text{for $\mu$-a.e.~$\omega \in \Omega$.}
\end{equation}
Oseledets observed that the angles $\theta_{ij}$ between different subspaces $\mathbf{E}_i\neq \mathbf{E}_j$ of the Oseledets splitting (or, more generally, between transverse sums of Oseledets subspaces) are tempered functions. 
This is a simple consequence of the log-integrability hypothesis \eqref{e.log-integrable_cocycle}, together with the following geometrical fact:
\begin{equation}\label{e.paralleloram}
\big| \log \sin \theta_{ij}(T \omega) - \log \sin \theta_{ij}(\omega) \big| 
\le \log \|F(\omega)\| + \log \|F(\omega)^{-1}\| \, . 
\end{equation}
In fact, temperedness properties play a fundamental role in Pesin's theory of nonuniform hyperbolicity: see \cite{barreira}, \cite[Supplement]{KH}.

We say that a positive measurable function $f$ is \emph{log-integrable} if $\log f \in L^1(\mu)$.
Note that every log-integrable function is tempered, as an immediate consequence of the ergodic theorem. 
Thus, the following question arises: \emph{are the angles between Oseledets spaces actually log-integrable?} 
To the best of our knowledge, this question has not been addressed in print before.

It turns out that the question above has a negative answer. In fact, we will exhibit an example in the class of \emph{one-step cocycles}, which model products of random independent identically distributed matrices.

From now on, our discussion will be restricted to dimension $d=2$.

\subsection{Failure of log-integrability of the angle for products of random i.i.d.~matrices}\label{ss.onestep}

Suppose our dynamics is a two-sided Bernoulli shift, which we denote $\sigma$ instead of $T$.
This means that the probability space $(\Omega,\mathcal{S},\mu)$ is an infinite product $(A,\mathcal{T},\pi)^{\otimes \Z}$, where $(A,\mathcal{T},\pi)$ is another probability space, and $\sigma \colon \Omega \to \Omega$ is the shift map $(\sigma \omega)_n = \omega_{n+1}$.

Consider a measurable two-dimensional linear cocycle $(\sigma, F)$ over the shift. The cocycle is called \emph{one-step} if the function $F \colon \Omega \to \mathrm{GL}_2(\R)$ can be written in the form $F = \Phi \circ p$, where $p \colon \Omega \to A$ is the projection on the zeroth coordinate and $\Phi \colon A \to \mathrm{GL}_2(\R)$ is some Borel measurable function.

For a one-step cocycle $(\sigma, F)$ as above, if the point $\omega$ is sampled randomly according to Bernoulli measure, the factors of the product $F^{(n)}(\omega) = F(\sigma^{n-1} \omega) \cdots F(\omega)$ are independent and identically distributed matrices. The common distribution is the push-forward $\nu \coloneqq F_* \mu = \Phi_* \pi$, which is a Borel probability measure on the group $\mathrm{GL}_2(\R)$. 

Note that the measure $\nu$ contains all the relevant information about the one-step cocycle $(\sigma,F)$, since we would not lose anything in assuming that $A = \mathrm{GL}_2(\R)$, $\pi = \nu$, and $\Phi = \mathrm{id}$. 

For one-step cocycles, the log-integrability condition is expressed in terms of the matrix distribution $\nu$ as
\begin{equation}
\int_{\mathrm{GL}_2(\R)} \log \max \big( \|g\|, \|g^{-1}\| \big) \, d\nu(\omega) < \infty \, .
\end{equation}
The left-hand side will be called the \emph{first moment} of $\nu$.

Thus, if the first moment is finite, the Lyapunov exponents of the one-step cocycle are well-defined and finite; we denote them as $\lambda_1(\nu)$ and $\lambda_2(\nu)$ to emphasize their exclusive dependence on $\nu$. 
A classical theorem of Furstenberg says that $\lambda_1(\nu) \neq \lambda_2(\nu)$ except in a few exceptional situations that can be explicitly described (see \cite{ledrappier,DF}).

As announced above, non-log-integrable Oseledets angles are possible in this setting:

\begin{theorem}\label{t.nonintegrable}
There exists a probability measure on $\mathrm{GL}_2(\R)$ with finite first moment  such that the associated i.i.d.~product (one-step cocycle) has distinct Lyapunov exponents $\lambda_1(\nu) \neq \lambda_2(\nu)$ and the angle between the Oseledets directions is not log-integrable, that is,
\begin{equation}
\int_{\Omega}\left|\log\sin \angle(\mathbf{E}_1(\omega),\mathbf{E}_2(\omega)) \right|\, d\mu(\omega) = \infty \, .
\end{equation}
\end{theorem}

The example is entirely explicit and can be found in Subsection~\ref{ss.iid_proofs}.

\subsection{A criterion for log-integrability}

Our next result provides a sufficient criterion for the log-integrability of the Oseledets angles, in the same random i.i.d.~setting. 

Let us say that a Borel probability measure $\nu$ on $\mathrm{GL}_2(\R)$ has \emph{finite second moment} if
\begin{equation}
\int_{\mathrm{GL}_2(\R)} \big[ \log \max ( \|g\|, \|g^{-1}\| ) \big]^2 \, d\nu(\omega) < \infty \, .
\end{equation}

\begin{theorem}\label{t.integrable}
Suppose $\nu$ is a probability measure on  $\mathrm{GL}_2(\R)$ with finite second moment such that the associated i.i.d.~product (one-step cocycle) has distinct Lyapunov exponents $\lambda_1(\nu) \neq \lambda_2(\nu)$.
Then the angle between the Oseledets direction is log-integrable, that is,
\begin{equation}
\int_{\Omega}\left|\log\sin \angle(\mathbf{E}_1(\omega),\mathbf{E}_2(\omega)) \right|\, d\mu(\omega) < \infty \, .
\end{equation}
\end{theorem}

The proof is also given in Subsection~\ref{ss.iid_proofs}.
In fact, the ``generic'' case in Theorem~\ref{t.integrable} follows immediately from results of \cite{benoist-quint}, and our proof consists of an analysis of the ``exceptional'' cases.

\subsection{Flexibility of Oseledets data for measurable cocycles}

We now come again to the general setting of measurable linear cocycles over ergodic automorphisms. Let us  assume that the underlying probability space is a Lebesgue space, and is non-atomic (since otherwise the dynamics reduces to a single periodic orbit). 

We ask whether Theorem~\ref{t.integrable} can be extended to this setting. 
We will show that the answer is negative: no integrability condition on the cocycle guarantees log-integrability of the Oseledets angles. 

In fact, we will show more: there is no restriction on the joint distribution of the Oseledets spaces.
Let us denote by $\mathsf{X}$ the product of two copies of projective space $\R\mathrm{P}^1$ minus the diagonal, that is,
\begin{equation}\label{e.X}
\mathsf{X} \coloneqq \big\{(x_1,x_2) \in \R\mathrm{P}^1 \times \R\mathrm{P}^1  \st x_1 \neq x_2 \big\} \, .
\end{equation}

\begin{theorem}\label{t.flexible}
Let \(T\) be an ergodic automorphism of a non-atomic Lebesgue space \((\Omega,\mathcal{S},\mu)\).  Let \(N:\GL_2(\R) \to \R\) be a locally bounded function such that
\begin{equation}\label{e.lowerN}
  N(g) \ge \log \max(\|g\|,\|g^{-1}\|).
\end{equation}
Let \(r_1 > r_2\) be real numbers and let \(\eta\) be a Borel probability measure on $\mathsf{X}$. 
Then there exists a measurable map \(F:\Omega \to \GL_2(\R)\) such that
\begin{equation}\label{e.super-integrability}
  \int_{\Omega} N(F(\omega)) \, d\mu(\omega) < \infty,
\end{equation}
the associated Lyapunov exponents are \(r_1,r_2\), and the Oseledets spaces \(\E_1,\E_2\) have joint distribution \(\eta\), that is, the push-forward of \(\mu\) under the map  \(\omega \mapsto (\E_1(\omega),\E_2(\omega))\) is \(\eta\). 
\end{theorem}

Thus, in the setting of two-dimensional measurable cocycles,  Oseledets data is \emph{flexible} in the sense of Katok's flexibility program \cite{BKRH}.

\subsection{The case of bounded cocycles}

In the setting of the previous subsection, we say the map $F \colon \Omega \to \mathrm{GL}_2(\R)$ is \emph{bounded} if its image $F(\Omega)$ is a relatively compact subset of $\mathrm{GL}_2(\R)$, and \emph{essentially bounded} if \(\supp(F_*\mu)\) is compact. In the latter case, we can alter $F$ on a zero-measure subset so that it becomes bounded.  

We now consider the flexibility problem in the class of bounded $\mathrm{GL}_2(\R)$-valued cocycles.  
There is an obvious restriction, 
which we will now describe.

We say that a probability measure on the space $\mathsf{X}$ (defined in \eqref{e.X}) has \emph{unbounded gap} if for every $\epsilon>0$, there exists two positive measure sets $A$ and $B \subseteq \mathsf{X}$ such that $\eta(A)>0$, $\eta(B)>0$, $\eta(A \cup B) = 1$, and
\begin{equation}
\sin \angle(x_1,x_2) < \epsilon \sin\angle(y_1,y_2) \quad \text{for all $(x_1,x_2)\in A$ and $(y_1,y_2)\in B$.}
\end{equation}
Otherwise, we say that $\eta$ has \emph{bounded gap}.

The distribution of the Oseledets angles of a bounded \(\GL_2(\R)\)-valued cocycle has bounded gap: this follows from inequality~\eqref{e.paralleloram} and the assumption that $T$ is ergodic.
We show that this is the only restriction.

\begin{theorem}\label{t.boundedflexible}
Let \(T\) be an ergodic automorphism of a non-atomic Lebesgue space \((\Omega,\mathcal{S},\mu)\).  
Let \(r_1 > r_2\) be real numbers and let \(\eta\) be a Borel probability measure on $\mathsf{X}$ and with bounded gap. 
Then there exists a bounded measurable map \(F:\Omega \to \GL_2(\R)\) such that the cocycle $(T,F)$ has Lyapunov exponents \(r_1,r_2\), and the Oseledets spaces \(\E_1,\E_2\) have joint distribution \(\eta\). 
\end{theorem}

\subsection{Organization of the paper}

Section~\ref{s.iid} contains the proofs of Theorems \ref{t.nonintegrable} and \ref{t.integrable}, while the proofs of Theorems~\ref{t.flexible} and \ref{t.boundedflexible} are given in the (entirely independent) Section~\ref{s.general}. The final section contains a brief discussion of directions for further research.


\section{The example and the criterion}\label{s.iid}

\subsection{Preparations}

We present some simple lemmas that will be used in the proofs of both Theorems \ref{t.nonintegrable} and \ref{t.integrable}.

\begin{lemma}[Triangular cocycles] \label{l.triangular}
Let $T$ be an ergodic automorphism of a probability space $(\Omega,\mathcal{S},\mu)$.
Let $F \colon \Omega \to \mathrm{GL}_2(\R)$ be a measurable function of the form
\begin{equation}
F(\omega) = \begin{pmatrix} a(\omega) & b(\omega) \\ 0 & 1 \end{pmatrix}
\end{equation}
such that $\log|a|$ and $\log^{\smallplus} |b|$ belong to $L^1(\mu)$, and $\int \log|a| \, d\mu < 0$.
Then the Lyapunov exponents of the cocycle $(T,F)$ are $\lambda_1 = 0$ and $\lambda_2 = \int \log|a| \, d\mu$, and the corresponding Oseledets spaces are, for $\mu$-a.e.~$\omega \in \Omega$,
\begin{equation}
\mathbf{E}_1(\omega) = \mathrm{span} \, \begin{pmatrix} X(\omega) \\ 1 \end{pmatrix} \quad \text{and} \quad
\mathbf{E}_2(\omega) = \mathrm{span} \, \begin{pmatrix} 1         \\ 0 \end{pmatrix} \, ,
\end{equation}
where
\begin{equation}\label{e.X_series}
X(\omega) \coloneqq \sum_{n=0}^\infty a(T^{-1} \omega) a(T^{-2} \omega) \cdots a(T^{-n} \omega) b(T^{-n-1} \omega) \, .
\end{equation}
\end{lemma}

The proof is straightforward. For closely related results, see \cite[Sec.~5]{kifer} and \cite[p.~161]{arnold}. 

\begin{lemma}[Weierstrass product inequalities]\label{l.Weierstrass}
For any convergent series $\sum a_n$  with terms $a_n \in [0,1]$, we have
\begin{equation}\label{e.quasi_de_Morgan}
\frac{\sum a_n}{1 + \sum a_n} \le 1 - \prod (1-a_n) \le \sum a_n \, .
\end{equation}
\end{lemma}

See \cite[{\S}38]{Bromwich}. 
We remark that the lower bound in \eqref{e.quasi_de_Morgan} can be improved to $1 - e^{-\sum a_n}$, but the one above suffices for our purposes.

\begin{lemma}\label{l.double-edged}
Let $(A,\mathcal{T},\pi)$ be a probability space 
and let $\psi \colon A \to [0,\infty)$ be integrable.
On the product space $(\Omega,\mathcal{S},\mu) = (A,\mathcal{T},\pi)^{\otimes \Z}$, consider the function $Y \colon \Omega \to [0,\infty]$ defined by:
\begin{equation}\label{e.def_Y}
Y(\omega) \coloneqq \sup_{n \ge 0} \left[ \psi(\omega_{-n-1}) - n \right] \, , \quad \text{where } \omega = (\omega_n)_{n \in \Z} \, .
\end{equation}
Then $Y$ is $\mu$-integrable if and only if $\psi^2$ is $\pi$-integrable.
\end{lemma}

\begin{proof}
Note that $Y \ge 0$. 
By the layer cake formula,
\begin{equation}
\int Y \, d\mu = \int_0^\infty \mu [Y \ge t] \, dt \, ,
\end{equation}
where we are using the probabilist's notation $[Y \ge t] \coloneqq \{\omega \in \Omega \st Y(\omega) \ge t \}$.
By the Maclaurin--Cauchy test, integrability of $Y$ is equivalent to convergence of the series $\sum b_k$, where $b_k \coloneqq \mu[Y \ge k]$.
Let also 
$a_k \coloneqq \pi[\psi \ge k]$ and $C \coloneqq \int \psi \, d\pi$.
By assumption, $C<\infty$.
Another application of the layer cake formula yields $\sum_{k=1}^\infty a_k \le C$.

Note that $[Y < k] = \bigcap_{n=0}^\infty [\psi(\omega_{-n-1})<n+k]$, an intersection of independent events whose probabilities are
\begin{equation}
\mu[\psi(\omega_{-n-1})<n+k] = \mu[\psi(\omega_0)<n+k] = 1 - a_{n+k} \, . 
\end{equation}
Therefore,
\begin{equation}
b_k = 1 - \mu[Y<k] = 1 - \prod_{n=0}^\infty (1 - a_{n+k}) = 1 - \prod_{j=k}^\infty (1-a_j) \, .
\end{equation}
So Lemma~\ref{l.Weierstrass} gives
\begin{equation}
\frac{1}{1+C} \sum_{j=k}^\infty a_j \le b_k \le \sum_{j=k}^\infty a_j \, .
\end{equation}
In particular, convergence of the series $\sum b_k$ is equivalent to the convergence of the double series $s \coloneqq \sum_{k=1}^\infty\sum_{j=k}^\infty a_j$.
The latter can be rewritten as:
\begin{equation}
s = \sum_{j=1}^\infty j a_j 
= \sum_{j=1}^\infty (1+\cdots+j) (a_j - a_{j+1})  \\ 
= \sum_{j=1}^\infty \frac{j(j+1)}{2} \pi[j \le \psi < j+1]  \, ,
\end{equation}
or equivalently,
\begin{equation}
s = \int q \big( \lfloor\psi(\alpha) \rfloor \big) \, d\pi(\alpha) \, , \quad \text{where} \quad 
q(x) \coloneqq \frac{x(x+1)}{2} \, .
\end{equation}
Recalling the underlying assumption $\int \psi \, d\pi < \infty$, it becomes clear that 
\begin{equation}
s < \infty  \quad \Leftrightarrow \quad 
\int q \circ \psi \, d\pi < \infty  \quad \Leftrightarrow \quad 
\int \psi^2 \, d\pi < \infty  \, .
\end{equation}
We have already seen that $\int Y \, d\mu < \infty$ $\Leftrightarrow$ $s<\infty$, and so the proof is complete. 
\end{proof}

\subsection{Proofs of the theorems} \label{ss.iid_proofs}

Now we put ourselves in the setting of one-step cocycles $(\sigma,F)$, as explained in Subsection~\ref{ss.onestep}. 
So $\sigma$ is the Bernoulli shift on the space $(\Omega,\mathcal{S},\mu) = (A,\mathcal{T},\pi)^{\otimes \Z}$, and $F \colon \Omega \to \mathrm{GL}_2(\R)$ is such that $F(\omega)$ only depends on the zero-th coordinate~$\omega_0$.

\begin{proof}[Proof of Theorem~\ref{t.nonintegrable}]
Let $(\sigma,F)$ be any one-step cocycle where $F$ is of the form
\begin{equation}
F(\omega) = \begin{pmatrix} e^{-1} & e^{\psi(\omega_0)} \\ 0 & 1 \end{pmatrix}
\end{equation}
for a non-negative function $\psi$ such that
\begin{equation}
\int \psi \, d\pi < \infty \quad \text{and} \quad
\int \psi^2 \, d\pi = \infty \, .
\end{equation}
By Lemma~\ref{l.triangular}, the cocycle has Lyapunov exponents $\lambda_1 = 0$, $\lambda_2 = -1$, and the corresponding Oseledets directions are spanned by 
the vectors $(X(\omega),1)^\top$ and $(1,0)^\top$,
where
\begin{equation}
X(\omega) = \sum_{n=0}^\infty e^{\psi(\omega_{-n-1}) - n} \quad \text{for $\mu$-a.e.\ $\omega = (\omega_n)_{n \in \Z}$.}
\end{equation}
Then $X(\omega) \ge e^{Y(\omega)}$, where $Y$ is defined as in \eqref{e.def_Y}.
A direct application of Lemma~\ref{l.double-edged} shows that $Y$ is non-integrable.
Thus, $\log X$ is also non-integrable.
In terms of the Oseledets angle $\theta \coloneqq \angle(\mathbf{E}_1, \mathbf{E}_2)$, we have $X = \cot \theta$.
It follows that $\log \theta$ is non-integrable. 
\end{proof}

\begin{proof}[Proof of Theorem~\ref{t.integrable}]
Let $\nu$ be a probability measure on $\mathrm{GL}(2,\R)$ with finite second moment and such that $\lambda_1(\nu) > \lambda_2(\nu)$.
Let $\Gamma_\nu$ be semigroup generated by the support of $\nu$.
As a first case, assume that $\Gamma_\nu$ is \emph{strongly irreducible}, that is, its action on projective space admits no finite invariant sets other than the empty set.
The hypothesis that \(\lambda_1(\nu) > \lambda_2(\nu)\) implies that \(\Gamma_\nu\) contains a proximal element (that is, an element with two real eigenvalues of distinct moduli): see \cite[Lemma~4.1]{BQbook}.
This allows us to apply  the key result of \cite[Proposition 4.5]{benoist-quint}, 
which tells us that the function
\begin{equation}
\xi \in \R \mathrm{P}^1 \mapsto \int \left| \log \sin \angle (\mathbf{E}_1(\omega) , \xi) \right| \, d\mu(\omega)
\end{equation}
is continuous.
Since the projective space $\R \mathrm{P}^1$ is compact, the function is bounded.

Given $\omega = (\omega)_{n \in \Z} \in \Omega$, write $\omega = (\omega_{\smallminus},\omega_{\smallplus})$, where $\omega_{\smallminus} = (\omega)_{n<0}$ and $\omega_{\smallplus} = (\omega)_{n \ge 0}$.
For one-step cocycles, the first Oseledets direction $\mathbf{E}_1(\omega)$ depends only on $\omega_{\smallminus}$, while the second direction $\mathbf{E}_2(\omega)$ depends only on $\omega_{\smallplus}$. So we can denote them as  $\mathbf{E}_1(\omega_{\smallminus})$ and $\mathbf{E}_2(\omega_{\smallplus})$. Let $\mu_{\smallpm}$ denote the push-forward of $\mu$ under the projection $\omega \mapsto \omega_{\smallpm}$. Then $\mu$ coincides with the product measure $\mu_\smallminus \otimes \mu_\smallplus$. 
Therefore, by Fubini--Tonelli theorem,
$\int \left| \log \sin \angle (\mathbf{E}_1, \mathbf{E}_2) \right| \, d\mu$ equals the double integral
\begin{equation}
\int \int \left| \log \sin \angle (\mathbf{E}_1(\omega_{\smallminus}) , \mathbf{E}_2(\omega_{\smallplus})) \right| \, d\mu_{\smallminus}(\omega_{\smallminus}) \, d\mu_{\smallplus}(\omega_{\smallplus}) 
\end{equation}
which is finite since the inner integral is bounded. 
This concludes the proof of Theorem~\ref{t.integrable} in the strongly irreducible case.

Next, we must deal with the cases where \(\Gamma_\nu\) permutes a finite set \(S\) of one-dimensional subspaces of \(\R^2\).

If \(\varhash S \ge 3\), then a cross-ratio argument shows that the projection of \(\Gamma_\nu\) to \(\mathrm{PGL}_2(\R)\) is finite.  This implies that \(\Gamma_\nu\) consists of linear conformal mappings with respect to some inner product on \(\R^2\), which is incompatible with the hypothesis \(\lambda_1(\nu) > \lambda_2(\nu)\).

If \(\varhash S = 2\), then the hypothesis \(\lambda_1(\nu) > \lambda_2(\nu)\) implies that \(\lbrace \mathbf{E}_1(\omega), \mathbf{E}_2(\omega)\rbrace = S\) \(\mu\)-a.e., and in particular the angle between Oseledets subspaces is log-integrable.

It remains to consider the case \(\varhash S = 1\).   Taking inverses we can assume that the stable subspace \(\mathbf{E}_2(\omega)\) belongs to $S$ for \(\mu\)-a.e.~\(\omega \in \Omega\).   With a change of basis we further assume that all \(\nu\)-almost every matrix is upper triangular.  Multiplying the cocycle matrices by a log-integrable factor, we can assume that the lower right entry is always $1$. That is, the function $F$ is of the form:
\begin{equation}
F(\omega) = \begin{pmatrix} a(\omega_0) & b(\omega_0) \\ 0 & 1 \end{pmatrix}  \, ,
\end{equation}
where $\log|a|$ and $\log|b|$ are square-integrable and $\int \log |a| \, d\pi < 0$.
By Lemma~\ref{l.triangular}, the cocycle has Lyapunov exponents $\lambda_1 = 0$ and $\lambda_2 = \int \log |a| \, d\pi$, and the corresponding Oseledets directions are spanned by 
the vectors $(X(\omega),1)^\top$ and $(1,0)^\top$,
where
\begin{equation}
X(\omega) 
= \sum_{n=0}^\infty a(\omega_{-1}) a(\omega_{-2}) \cdots a(\omega_{-n}) b(\omega_{-n-1}) \, .
\end{equation}
Writing 
$\phi(\omega_0) \coloneqq - \log |a(\omega_0)|$, 
$\psi(\omega_0) \coloneqq \log |b(\omega_0)|$, we have:
\begin{equation}
X(\omega) = \sum_{n=0}^\infty \pm e^{S_n(\omega)} \quad \text{where} \quad
S_n(\omega) \coloneqq \psi(\omega_{-n-1}) - \sum_{i=1}^n \phi(\omega_{-i}) \, .
\end{equation}
We can bound the last expression as:
\begin{equation}
S_n(\omega) \le  
\underbrace{\Big[ \psi^{\smallplus}(\omega_{-n-1}) -cn \Big]}_{\eqcolon Y_n(\omega)} +
\underbrace{\Big[ 2cn - {\textstyle\sum_{i=1}^n \phi(\omega_{-i}) } \Big]}_{\eqcolon Z_n(\omega)}  
-cn \, ,
\end{equation}
where $c \coloneqq \frac{1}{3} \int \phi \, d\pi > 0$.
Consider the following quantities:
\begin{equation}
Y(\omega) \coloneqq \sup_{n \ge 0} Y_n(\omega) \quad \text{and} \quad
Z(\omega) \coloneqq \sup_{n \ge 0} Z_n(\omega) \, .
\end{equation}
An application of Lemma~\ref{l.double-edged} to the square-integrable function $\psi^{\smallplus}/c$ shows that $Y$ is integrable. 
On the other hand, the sequence $(Z_n(\omega))$ forms a random walk with negative drift and square-integrable (i.i.d.) increments. Therefore, by a theorem of Kiefer and Wolfowitz \cite[Theorem~5]{KieferWolfowitz} (or \cite[Corollary~3, p.~397]{Chow}), the supremum $Z(\omega)$ is integrable.

We have $S_n(\omega) \le Y(\omega) + Z(\omega) -cn$ and thus
\begin{equation}
\log |X (\omega)| \le \log \sum_{n=0}^\infty e^{S_n(\omega)} \le  Y(\omega) + Z(\omega) + \text{constant} \, .
\end{equation}
It follows that $\log^{\smallplus} |X| = \log^{\smallplus} \cot \angle(\mathbf{E}_1, \mathbf{E}_2)$ is integrable, as we wanted to show. 
\end{proof}

\begin{remark}
It is not possible to find $\nu$ in the context of Theorem~\ref{t.nonintegrable} such that $\Gamma_\nu$ is strongly irreducible. This follows from a recent result of P\'eneau \cite[Corollary~1.14]{Peneau}. We thank Cagri Sert for this observation. 
\end{remark}

\begin{remark}
Under other types of moment conditions on $a$ and $b$, information on distribution of the corresponding random variable $X$ defined by \eqref{e.X_series} was obtained by Kesten and Goldie: see \cite{BDZ} and references therein. We thank Cagri Sert for telling us about this. 
\end{remark}


\section{Flexibility of Oseledets data}\label{s.general}

In this section we prove Theorems~\ref{t.flexible} and \ref{t.boundedflexible}.  
We start we a few preparations.

\subsection{Construction of functions with prescribed distributions and small average costs} 

Recall that a \emph{Polish space} is a separable completely metrizable topological space.


\begin{lemma}[Function with prescribed distribution]\label{l.basic}
If \((\Omega,\mathcal{S},\mu)\) is a non-atomic Lebesgue probability space
and \(\eta\) a Borel probability measure on a Polish space \(X\),
then there exists a measurable map \(f:\Omega \to X\) with distribution \(\eta\), that is, \(f_*\mu= \eta\).
\end{lemma}

\begin{proof}
We can assume that $\Omega$ is the unit interval, $\mathcal{S}$ is the Lebesgue $\sigma$-algebra, and $\mu$ is Lebesgue measure. If $X = \R$, the statement is a standard construction: see \cite[p.~429, Proposition~2, part~(2)]{Embrechts}. The case of a general Polish space follows from the Borel isomorphism theorem \cite[p.~99]{Srivastava}.
\end{proof}

\medskip

A \emph{cost function} on a Polish space $X$ is a nonnegative upper semicontinuous function $c \colon X \times X \to \R$ that vanishes on the diagonal. 

\begin{theorem}\label{t.lowcost}
Let $T$ be an ergodic automorphism of a non-atomic Lebesgue probability space $(\Omega,\mathcal{S},\mu)$.
Let $X$ be a Polish space, and let $\eta$ be a Borel probability measure on $X$. 
Let $c \colon X \times X \to [0,\infty)$ be a cost function. 
Then, for any $\epsilon>0$, there exists a measurable function $f \colon \Omega \to X$ such that 
$f_* \mu = \eta$ and 
\begin{equation}\label{e.lowcost}
\int_\Omega c(f(\omega),f(T \omega)) \, d\mu(\omega) < \epsilon \, .
\end{equation}
\end{theorem}

The conclusion of the theorem above implies that, for $\mu$-almost every $\omega$, the sequence of points $x_n \coloneqq f(T^n \omega)$ travels around the space $X$ achieving the prescribed distribution $\eta$ in such a way that $c(x_n,x_{n+1})$ (the cost of travel per unit of time) is small on average. As the proof will show, low cost travel can be achieved by moving slowly. 

\medskip

The proof of Theorem~\ref{t.lowcost} is a skyscraper construction. 
Let us recall the relevant facts.

Suppose $T$ is an ergodic automorphism of a probability space $(\Omega,\mathcal{S},\mu)$, and that $B \in \mathcal{S}$ has $\mu(B)>0$. By the Poincar\'e recurrence theorem, for $\mu$-almost every $\omega \in B$, there exists $k \ge 1$ such that $T^k \omega \in B$. Let $B_k$ denote the set of points of $B$ whose first return to $B$ occurs at time $k$, that is,
\begin{equation}\label{e.Bk}
B_k \coloneqq B \cap T^{-k}(B) \setminus \bigcup_{i=1}^{k-1} T^{-i}(B) \, .
\end{equation}
For each $k$, the sets $B_k, TB_k, \dots, T^{k-1} B_k$ are disjoint, and their union forms a \emph{tower} of height $k$ and measure $k \mu(B_k)$.
These towers form a disjoint collection, called the \emph{skyscraper} with base $B$.
By ergodicity, the skyscraper covers the whole space $\Omega$, except for a zero measure set. 
In particular, $\sum_{k=1}^\infty k \mu(B_k) = 1$; this statement is known as \emph{Kac's lemma}.

Assuming that $(\Omega,\mathcal{S},\mu)$ is a non-atomic Lebesgue space, a partial converse to the Kac's lemma is provided by the following theorem of Alpern and Prasad \cite{AP_paper} (or \cite[Theorem~A.1.4]{AP_book}): 
given any sequence $\pi = (\pi_1,\pi_2,\dots)$ of nonnegative numbers such that
\begin{equation} 
\sum_{k=1}^\infty \pi_k = 1 \quad \text{and} \quad \mathrm{gcd}\{k \st \pi_k \neq 0\} = 1 \, , 
\end{equation}
we can find a positive measure subset $B \subseteq \Omega$ such that the towers of the skyscraper with base $B$ have measures as specified by the sequence $\pi$. 
Equivalently, the sets \eqref{e.Bk} satisfy 
\begin{equation}\label{e.AP}
\mu(B_k) = \frac{\pi_k}{k} \quad \text{for every $k \ge 1$.}
\end{equation}
Note that the Alpern--Prasad theorem is an extension of the Rokhlin tower lemma. 

\begin{proof}[Proof of Theorem~\ref{t.lowcost}]
Since \(X\) is Polish, by \cite[p.~8, Theorem~1.3]{billingsley} the measure \(\eta\) is tight, meaning there is a countable union of compact sets which has full measure.  This implies we can write 
$\eta$ as a convex combination $\sum_{n=1}^\infty p_n \eta_n$ of compactly supported probability measures $\eta_n$. 
Let $K_n$ be the union of the supports of $\eta_1,\dots,\eta_n$.

We define 
\begin{equation}
C_n \coloneqq \max_{x,x' \in K_n}c(x,x') \, ,
\end{equation} 
which exists because \(c\) is upper semicontinuous and \(K_n\) is compact.  This is a non-decreasing sequence, because \(K_n \subseteq K_{n+1}\) for all \(n\).

Given \(\epsilon > 0\), we choose a sequence of positive numbers \(k_1 < k_2 < \cdots\) whose GCD is $1$ and is such that 
\begin{equation}\label{e.stoppedtimes}
\frac{C_n}{k_n} < \frac{\epsilon}{2}  \quad \text{for all \(n\).}
\end{equation}
Let \(\pi = (\pi_1,\pi_2,\ldots)\) be the probability vector defined as $\pi_{k_n} \coloneqq p_n$ and \(\pi_k \coloneqq 0\) if \(k \notin \lbrace k_1,k_2,\ldots\rbrace\).
Now we let \(B\) be the set given for \(\pi\) by the Alpern--Prasad theorem: the associated subsets $B_k$ defined in \eqref{e.Bk} have measures specified by \eqref{e.AP}.

We now let \(f:\Omega \to X\) be such that \((f|_{B_{k_n}})_{*}( \mu|_{B_{k_n}}) = p_n\eta_n/k_n\) for all \(n\) (using Lemma~\ref{l.basic}) and \(f(\omega) = f(T^{-1}\omega)\) if \(\omega \notin B\).
It is easy to check that \(f_*\mu = \eta\).

We are left to bound the integral in \eqref{e.lowcost}. 
We claim that the integrand can be bounded as follows:
\begin{gather}\label{e.cost_estimate}
c(f(\omega) , f(T\omega)) \le  g(\omega) + h(\omega) \, ,  \quad \text{where}
\\
g \coloneqq \sum_{n=1}^\infty C_n \mathbf{1}_{T^{-1}(B_{k_n})} \quad \text{and} \quad 
h \coloneqq \sum_{n=1}^\infty C_n \mathbf{1}_{T^{k_n-1}(B_{k_n})} \, .
\end{gather}
Indeed, if $\omega \not\in T^{-1}(B)$, then all terms in \eqref{e.cost_estimate} vanish.
On the other hand, for $\mu$-a.e.\ $\omega \in T^{-1}(B)$, there exist $n$ and $m$ such that $\omega \in T^{-1}(B_{k_n}) \cap T^{k_m-1}(B_{k_m})$. Therefore, $f(T\omega) \in K_n$ and $f(\omega)\in K_m$. It follows that $c(f(\omega) , f(T\omega)) \le C_{\max(m,n)} \le C_m + C_n$, proving \eqref{e.cost_estimate}.
On the other hand,
\begin{equation}
\int g \, d\mu = \int h \, d\mu
= \sum_{n=1}^\infty C_n \mu(B_{k_n})
= \sum_{n=1}^\infty \frac{C_n p_n}{k_n}
< \sum_{n=1}^\infty \frac{\epsilon p_n}{2} 
= \frac{\epsilon}{2} \, ,
\end{equation}
so we obtain $\int c(f(\omega) , f(T\omega)) \, d\mu(\omega)<\epsilon$.
\end{proof}

\subsection{Flexibility for general cocycles}

We now use Theorem~\ref{t.lowcost} to prove our first theorem on flexibility of the Oseledets data.

\begin{proof}[Proof of Theorem~\ref{t.flexible}]
We consider
\begin{equation}
  \widetilde{\mathsf{X}} \coloneqq \left\lbrace (u_1,u_2) \in \R^2\times \R^2: \|u_1\| = \|u_2\| = 1, \  u_1 \neq \pm u_2\right\rbrace,
\end{equation}
together with the natural four-to-one projection \(\pi: \widetilde{\mathsf{X}} \to \mathsf{X}\).
Given $\tilde{x} = (u_1,u_2),\tilde{y} = (v_1,v_2) \in \widetilde{\mathsf{X}}$, let $\widetilde{\Phi}(\tilde{x},\tilde{y})$ be the unique element of $\mathrm{GL}_2(\R)$ that maps $u_1$ to $v_1$ and $u_2$ to $v_2$. 
We choose a measurable map \(\rho: \mathsf{X} \to \widetilde{\mathsf{X}}\) such that \(\pi\circ \rho\) is the identity on \(\mathsf{X}\), and define \(\Phi:\mathsf{X} \times \mathsf{X} \to \GL_2(\R)\) by \(\Phi\coloneqq \widetilde{\Phi} \circ (\rho\times \rho)\), that is, \(\Phi(x,y) \coloneqq \widetilde{\Phi}(\rho(x),\rho(y))\).

Next, we choose two functions $\psi_1, \psi_2 \colon \mathsf{X} \to \R$ which are continuous, have bounded support, and have averages $\int \psi_j \, d\eta = r_j$. For each $x = (x_1,x_2) \in \mathsf{X}$, let $\Psi(x)$ be the matrix with eigenspaces $x_1$, $x_2$ and corresponding eigenvalues $e^{\psi_1(x)}$, $e^{\psi_2(x)}$. Then $\Psi:\mathsf{X} \to \GL_2(\R)$ is a continuous function which equals the identity matrix outside a compact set.
Define \(\Upsilon: \mathsf{X} \times \mathsf{X} \to \GL_2(\R)\) as \(\Upsilon(x,y) \coloneqq \Phi(x,y)\Psi(x)\).

We assume without loss of generality that \(N\) is upper semicontinuous (replacing \(N\) by its \(\limsup\) at each point if necessary).
Define a function \(c:\mathsf{X} \times \mathsf{X} \to [0,+\infty)\) by
\begin{equation}\label{e.def_c}
c(x,y) \coloneqq 
\max\left\lbrace N(\widetilde{\Phi}(\tilde{x},\tilde{y})\Psi(x)) \st (\tilde{x},\tilde{y}) \in (\pi\times \pi)^{-1}(x,y))\right\rbrace \, .
\end{equation}
Observe that \(c\) is upper semicontinuous, because it is locally the maximum of finitely many upper semicontinuous functions. 
We are now allowed to apply Theorem \ref{t.lowcost}: note that, being an open subset of a Polish space, \(\mathsf{X}\) is itself Polish.
We obtain a measurable function \(f:\Omega \to \mathsf{X}\) such that \(f_*\mu = \eta\) and \(\int_{\Omega}c(f(\omega),f(T\omega)) \, d\mu(\omega) < \infty \). (The integral can actually be made small, but we will not need this). 

Let \(F(\omega) \coloneqq \Upsilon(f(\omega),f(T \omega))\) for all \(\omega\).  
Note that, by the definition of $c$, we have $N(\Upsilon(x,y)) \le c(x,y)$ for every $(x,y) \in \mathsf{X} \times \mathsf{X}$.
It follows that
\begin{equation}
\int N(F(\omega)) \, d\mu(\omega) \le \int_\Omega c(f(\omega),f(T \omega)) \, d\mu(\omega)
< \infty \, ,
\end{equation}
that is, requirement \eqref{e.super-integrability} is met. 
In particular, by \eqref{e.lowerN}, the cocycle \((T,F)\) has well-defined Lyapunov exponents. We will show that these exponents are \(r_1 , r_2\) and that the corresponding Oseledets subspaces are \(\E_1,\E_2\), where \(f(\omega) = (\E_1(\omega),\E_2(\omega))\).

It follows from the definitions above that,
for all $j \in \{1,2\}$ and $n>0$,
\begin{equation}
\frac{1}{n}\log \big\|F^{(n)}(\omega)|_{ \E_j(\omega)}\big\| =
\frac{1}{n} \sum_{i = 0}^{n-1} \psi_j(f(T^i \omega)) \, .
\end{equation}
By the ergodic theorem, as  $n \to +\infty$  the quantity above converges almost everywhere to $\int \psi_j \circ f \, d\mu = \int \psi_j \, d\eta = r_j$.
A similar argument also applies to $n \to -\infty$, yielding the same limit $r_j$. Therefore, \(\E_j\) is the Oseledets subspace for the Lyapunov exponent \(r_j\), as announced.
\end{proof}

\subsection{Construction of functions with prescribed distributions and small average costs} 

This subsection contains a preliminary result for the proof of Theorem~\ref{t.boundedflexible}.

Recall a \emph{cost function} on a Polish space $X$ is a nonnegative upper semicontinuous $c \colon X \times X \to \R$ that vanishes on the diagonal. 
We say that a cost function $c$ is \emph{symmetric} if $c(y,x) = c(x,y)$ for all $x,y \in X$. 

Let $\eta$ be a Borel probability measure on $X$ and let $b>0$.
We say that a symmetric cost function $c$ \emph{fits the budget $b$} with respect to~$\eta$ if
for any two Borel sets $A$, $B$ such that $\eta(A)>0$, $\eta(B)>0$, and $\eta(A \cup B) = 1$, we have
\begin{equation}
	\inf_{x \in A, \,  y \in B} c(x,y) < b \, .
\end{equation}

\begin{lemma}\label{l.March}
Let $\eta$ be a Borel probability measure on a Polish space $X$.
Let $c$ be a symmetric cost function on $X$ which fits a budget $b$ with respect to~$\eta$.
Then there exists a sequence $(E_n)$ of pairwise disjoint relatively compact subsets of $X$ such that $\eta(\bigsqcup E_n) = 1$ and for every $n$ we have $\eta(E_n)>0$ and 
\begin{equation}
	\sup_{x,y \in \bar{E}_n \cup \bar{E}_{n+1}} c(x,y) < b  \, . 
\end{equation}
\end{lemma}

\begin{proof}
Let $b>0$ be the given budget for the symmetric cost function $c$.
Every point of $X$ has a compact neighborhood $U$ such that $c|_{U \times U} < b$.
Since $\eta$ is tight, we can find a sequence $(U_n)$ of such neighborhoods such that $c|_{U_n \times U_n} < b$ for each $n$ and $\eta(\bigcup U_n) = 1$.
Define sets $A_n \coloneqq U_n \setminus \bigcup_{j=0}^{n-1} U_j$. Then the $A_n$'s are relatively compact, pairwise disjoint, their union has full measure, and $c|_{\bar{A}_n \times \bar{A}_n} < b$ for each $n$. 
Modifying the sequence if necessary, we can assume that $\eta(A_n)>0$ and $\supp \eta|_{\bar{A}_n} = \bar{A}_n$ for each $n$.

Consider the graph whose vertices are the $A_n$'s and there is an edge $A_n - A_m$ iff  
\begin{equation}
	\inf_{x \in A_n, \,  y \in A_m} c(x,y) < b  \, .
\end{equation}
Since $c$ fits the budget $b$, our graph is connected. 
Therefore, there exists a path $A_{n_1} \to A_{n_2} \to \cdots$ that visits all vertices. 

For each $i$, there exists $x_i \in A_{n_i}$ and $y_i \in A_{n_{i+1}}$ such that $c(x_i,y_i) < b$.
So we find compact neighborhoods $V_i$ and $W_i$ of $x_i$ and $y_i$, respectively, such that $c|_{V_i \times W_i} < b$.
Now let $E_1,E_2,\dots$ be the following sequence of sets:
\begin{equation}
A_{n_1}, \  A_{n_1} \cap V_1, \  A_{n_2} \cap W_1, \ 
A_{n_2}, \  A_{n_2} \cap V_2, \  A_{n_3} \cap W_2, \ 
A_{n_3}, \  \dots
\end{equation}
Then the sequence $(E_n)$ meets all requirements. 
\end{proof}

We now establish the following variant of Theorem~\ref{t.lowcost}:

\begin{theorem}\label{t.boundedcost}
Let $T$ be an ergodic automorphism of a non-atomic Lebesgue probability space $(\Omega,\mathcal{S},\mu)$.
Let $\eta$ be a Borel probability measure on a Polish space $X$.
Let $c$ be a symmetric cost function on $X$ which fits a budget $b$ with respect to~$\eta$.
Then there exists a measurable function $f \colon \Omega \to X$ such that $f_* \mu = \eta$ and 
\begin{equation}
c(f(\omega),f(T \omega)) < b \text{ for every \(\omega \in \Omega\).}
\end{equation}
\end{theorem}

In terms the previous travelling analogy, this time the cost of each of our moves cannot exceed a prescribed amount.

\begin{proof} 
It follows immediately from Lemma~\ref{l.March} that the probability measure~$\eta$ can be written as a convex combination $\sum_{n=0}^\infty p_n \eta_n$ of a sequence of probability measures $\eta_n$ of compact supports $K_n$ with the property that, for each~$n$,
\begin{equation}\label{e.lunchtime}
\sup_{x, \, y \in K_n \cup K_{n+1}} c(x,y) < b  \, . 
\end{equation}
Furthermore, we can assume that the sequence of weights is strictly decreasing, that is, $p_0>p_1>\dots$; indeed, it suffices to replace the sequence $\eta_0,\eta_1,\dots$ by 
\begin{equation}
\eta_0, \ 
\underbrace{\frac{\eta_1}{k_1} \ , \ \dots, \ \frac{\eta_1}{k_1}}_{\text{$k_1$ times}} \ , \ 
\underbrace{\frac{\eta_2}{k_2} \ , \ \dots, \ \frac{\eta_2}{k_2}}_{\text{$k_3$ times}} \ , \ 
\dots
\end{equation}
for some appropriate sequence $k_1,k_2,\dots$ and adjust the weights accordingly.

Define a sequence $(\pi_k)_{k \ge 1}$ as follows: 
\begin{alignat}{2}
\pi_1 		&\coloneqq p_0 - p_1  \, , \\ 
\pi_{2n+2}	&\coloneqq (n+1)(p_n-p_{n+1}) \, &\quad&\text{for all $n \ge 1$,} \\ 
\pi_k		&\coloneqq 0 \, &\quad&\text{if $k \not\in S \coloneqq \{1,4,6,8,\dots\}$}.
\end{alignat}
The fact that the sequence $(p_n)_{n \ge 0}$ is strictly decreasing guarantees that $\pi_k > 0$ for every $k \in S$.
Furthermore, 
\begin{equation}
\sum_{k \in S} \pi_k = \sum_{n=0}^\infty (n+1)(p_n-p_{n+1}) = \sum_{n=0}^\infty p_n  = 1 \, .
\end{equation}
Also, $\mathrm{gcd}(S) = 1$.
By the Alpern--Prasad theorem, there exists a positive measure set $B \subseteq \Omega$ such that, for each $k$, the $k$-th tower of the skyscraper with base $B$ has measure $\pi_k$.

For almost every $\omega \in \Omega$, the following quantity is well-defined:
\begin{equation}
\ell(\omega) \coloneqq \min \big\{|n| \st n\in \Z, \ T^n \omega \in B \cup T^{-1} B \big\} \, .
\end{equation}
If the two sided orbit of $\omega$ never hits $B$, let $\ell(\omega) \coloneqq 0$.
The function $\ell$ is measurable and clearly has the property
\begin{equation}\label{e.at_most_one}
|\ell(T\omega) - \ell(\omega)| \le 1 \quad \text{for all $\omega \in \Omega$.}
\end{equation}
We call $\ell(\omega)$ the \emph{label} of the point $\omega$.
Note that labels are constant on each component of each tower of the skyscraper with base $B$: see Figure~\ref{f.sky}.

\begin{figure}
	\begin{tikzpicture}[scale=0.6]
		\draw[very thick](0, 0)--(2, 0) node[near end, above]{$0$};
		\draw[very thick](3, 0)--(5, 0) node[near end, above]{$0$};
		\draw[very thick](3, 2)--(5, 2) node[near end, above]{$1$};
		\draw[very thick](3, 4)--(5, 4) node[near end, above]{$1$};
		\draw[very thick](3, 6)--(5, 6) node[near end, above]{$0$};
		\draw[very thick](6, 0)--(8, 0) node[near end, above]{$0$};
		\draw[very thick](6, 2)--(8, 2) node[near end, above]{$1$};
		\draw[very thick](6, 4)--(8, 4) node[near end, above]{$2$};
		\draw[very thick](6, 6)--(8, 6) node[near end, above]{$2$};
		\draw[very thick](6, 8)--(8, 8) node[near end, above]{$1$};
		\draw[very thick](6,10)--(8,10) node[near end, above]{$0$};
		\draw[->] (4,0.75)--(4,1.25);
		\draw[->] (4,2.75)--(4,3.25);
		\draw[->] (4,4.75)--(4,5.25);
		\draw[->] (7,0.75)--(7,1.25);
		\draw[->] (7,2.75)--(7,3.25);
		\draw[->] (7,4.75)--(7,5.25);
		\draw[->] (7,6.75)--(7,7.25);
		\draw[->] (7,8.75)--(7,9.25);
		\draw (10, 0) node{$...$};
		\draw (10, 2) node{$...$};
		\draw (10, 4) node{$...$};
		\draw (10, 6) node{$...$};
		\draw (10, 8) node{$...$};
		\draw (10,10) node{$...$};
		\draw (10,12) node{$...$};
		\draw[thick, decoration={brace, mirror}, decorate] (0,-1)--(2,-1)
			node[pos=0.5,anchor=north,yshift=-2mm]{$\pi_1$};
		\draw[thick, decoration={brace, mirror}, decorate] (3,-1)--(5,-1)
			node[pos=0.5,anchor=north,yshift=-2mm]{$\pi_4$};
		\draw[thick, decoration={brace, mirror}, decorate] (6,-1)--(8,-1)
			node[pos=0.5,anchor=north,yshift=-2mm]{$\pi_6$};
		\draw[thick, decoration={brace, mirror}, decorate] (11,-0.5)--(11,0.5)
			node[pos=0.5,anchor=west,xshift=2mm]{$B$};
	\end{tikzpicture}
	\caption{The skyscraper with base $B$, labeled according to the function $\ell(\mathord{\cdot})$.}\label{f.sky}
\end{figure}

Let $L_n \coloneq  \ell^{-1}(n)$ be the set of points with label $n$.
Let us compute $\mu(L_n)$.
Note that the first tower (of height $1$) is labeled $0$, half of the second tower (of height $4$) is labeled $0$, one third of the third tower (of height $6$) is labeled $0$, etc., so 
\begin{equation}
\mu(L_0) = \pi_1 + \frac{\pi_4}{2} + \frac{\pi_6}{3} + \cdots = (p_0-p_1) + (p_1-p_2) + (p_2-p_3) + \cdots = p_0 \, ,
\end{equation}
Similarly, for each label $n\ge 1$,
\begin{equation}
\mu(L_n) = \frac{\pi_{2n+2}}{n} + \frac{\pi_{2n+4}}{n+1} + \cdots = (p_n - p_{n+1}) + (p_{n+1} - p_{n+2}) + \cdots = p_n \, .
\end{equation}
That is, $\mu(L_n) = p_n$ for every $n \ge 0$.
Since the sets $L_n$ form a mod~$0$ partition of $\Omega$, using Lemma~\ref{l.basic} on each partition element, we can construct a measurable map $f \colon \Omega \to X$ such that $f_* (\mu|_{L_n}) = p_n \eta_n$ for each $n \ge 0$.
In particular, \(f_*\mu = \eta\).   
Recalling that $K_n = \supp(\eta_n)$, we can also assume that, for each $n$, if $\omega \in L_n$, then $f(\omega) \in K_n$. In this case, it follows from \eqref{e.at_most_one} that \(f(T \omega) \in K_{n-1} \cup K_n \cup K_{n+1}\) (where \(K_{-1} \coloneq \emptyset\)), and in particular \(c(f(\omega),f(T \omega)) < b \).
\end{proof}

\subsection{Flexibility for bounded cocycles}\label{ss.bounded}

Here we use the previous material to give the:

\begin{proof}[Proof of Theorem~\ref{t.boundedflexible}]
We retrace the steps of proof of Theorem~\ref{t.flexible}, with a few modifications. 
Define $\widetilde{\mathsf{X}}$ and $\widetilde{\Phi}$ exactly as before. 
We claim that for all \(\tilde{x} = (u_1,u_2), \tilde{y}=(v_1,v_2) \in \widetilde{\mathsf{X}}\), the singular values of the matrix $\widetilde{\Phi}(\tilde{x},\tilde{y})$ are
\begin{equation}\label{e.sing_val}
\frac{\sin (\theta'/2)}{\sin (\theta/2)} \quad \text{and} \quad
\frac{\cos (\theta'/2)}{\cos (\theta/2)} \, ,
\end{equation}
where $\theta = \angle(u_1,u_2)$, $\theta' = \angle(v_1,v_2)$.
To see this, consider the rhombus in $\R^2$ with sides of length $1$ and vertices $0$, $u_1$, $u_1+u_2$, $u_2$.
This rhombus is mapped by $\widetilde{\Phi}(\tilde{x},\tilde{y})$ to the rhombus  with sides of length $1$ and vertices $0$, $v_1$, $v_1+v_2$, $v_2$.
The diagonals of the first rhombus have lengths $\sin (\theta/2)$ and $\cos (\theta/2)$, and are mapped to 
the diagonals of the second rhombus of respective lengths $\sin (\theta'/2)$ and $\cos (\theta'/2)$.
Since diagonals of any rhombus are orthogonal, the singular value decomposition of $\widetilde{\Phi}(\tilde{x},\tilde{y})$ becomes apparent, and the claim follows. 

The next step is to choose a measurable map $\rho \colon \mathsf{X} \to \widetilde{\mathsf{X}}$ such that $\pi \circ \rho$ is the identity;  this time, we impose an extra condition:
\begin{equation}
	\rho(\mathsf{X}) \subseteq \left\{ (u_1,u_2) \in \widetilde{\mathsf{X}} \st \angle(u_1,u_2) \le \tfrac{\pi}{2} \right\} \, .
\end{equation}
(Note that, while the angle between two lines is by definition at most $\tfrac{\pi}{2}$, the angle between two nonzero vectors ranges from $0$ to $\pi$).
This precaution ensures that $\angle(u_1,u_2) = \angle(x_1,x_2)$ if $(u_1,u_2) = \rho(x_1,x_2)$.
Then, as before, let \(\Phi \coloneqq \widetilde{\Phi}\circ (\rho \times \rho)\). 
We claim that, for any $x = (x_1,x_2), y = (y_1,y_2) \in \mathsf{X}$, writing  $\theta = \angle(x_1,x_2)$, $\theta' = \angle(y_1,y_2)$, the matrix $g = \Phi(x,y)$ satisfies:
\begin{equation}\label{e.norms_bound}
\log \max \big( \|g\|, \|g^{-1}\| \big) = 
\left| \log \sin \frac{\theta'}{2} - \log \sin \frac{\theta}{2} \right| \, .
\end{equation}
Indeed, the singular values $\|g\|$ and $\|g^{-1}\|^{-1}$ are given by the previous expressions \eqref{e.sing_val} (up to switching). 
An application of the mean value theorem shows that
\begin{equation}
\left| \log \cos \frac{\theta'}{2} - \log \cos \frac{\theta}{2} \right|  \le 
\left| \frac{\theta'}{2} -  \frac{\theta}{2} \right| \le 
\left| \log \sin \frac{\theta'}{2} - \log \sin \frac{\theta}{2} \right| \, .
\end{equation}
Equality \eqref{e.norms_bound} follows. 

The functions $\Psi$ and $\Upsilon$ are defined exactly as before but, but the function $c$ is now defined as:
\begin{equation}
c(x,y) \coloneqq \left| \log \sin \frac{\angle(y_1,y_2)}{2} - \log \sin \frac{\angle(x_1,x_2)}{2} \right| \, ,
\end{equation}
which is the quantity \eqref{e.norms_bound}.
This $c$ is a symmetric cost function on $\mathsf{X}$.
The bounded gap assumption on the measure $\eta$ means that there exists $b>0$ such that $c$ fits a budget $b$ with respect to $\eta$.
Applying Theorem \ref{t.boundedcost}, we obtain a measurable function \(f:\Omega \to \mathsf{X}\) such that \(f_*\mu = \eta\) and 
\begin{equation}\label{e.Wednesday}
\sup_{\omega \in \Omega} c(f(\omega),f(T\omega)) < b \, .
\end{equation}

As before, let  
\begin{equation}
F(\omega) \coloneqq \Upsilon(f(\omega),f(T \omega)) = \Phi(f(\omega),f(T \omega)) \Psi(f(\omega)) \, .
\end{equation}
By its own definition, the function $\Psi$ is bounded (i.e., takes values into a compact subset of $\mathrm{GL}_2(\R)$). On the other hand, \eqref{e.Wednesday} means that the function $\omega \mapsto \Phi(f(\omega),f(T \omega))$ is bounded. Thus $F$ is bounded. 

The same argument as in the proof of Theorem~\ref{t.flexible} shows that the cocycle $(T,F)$ has the prescribed Lyapunov exponents $r_1,r_2$, and that the Oseledets spaces have joint distribution $\eta$.
\end{proof}

\section{Conclusions and comments} \label{s.conclusion}

In this article, we obtained results on the distribution of the angles between Oseledets subspaces, with a focus on the property of log-integrability.
For reasons of simplicity, we confined ourselves to dimension $2$, but we would expect similar results to hold in higher dimension. 

We showed that in the class of measurable linear cocycles, the distribution of the angles is completely arbitrary. On the other hand, for i.i.d.\ random products of matrices (one-step cocycles), we proved that the log-integrability of the angles follows from finiteness of the second moment, and we showed that this hypothesis is indispensable. 

A distinguishing feature of the i.i.d.\ setting  is that the distribution of the unstable direction $\mathbf{E}_1$ is a \emph{stationary measure} with respect to the matrix distribution $\nu$ (while the stable direction $\mathbf{E}_2$ is stationary with respect to the distribution of the inverse matrices). 

Stationary measures have been intensively studied.  
For example, they are known to be exact dimensional \cite{HS} (see also \cite{LL} and \cite{rapaport} for higher dimension). 
Their regularity properties have been studied in \cite{kaufmann} and also in \cite{GKM,Monakov}, which also apply to nonlinear dynamics.  
As for the dimension of the support of the stationary measures, see \cite{CJ}.   

In view of Theorem~\ref{t.flexible}, it would be interesting to understand whether some type of weak independence between the factors of a random matrix product would yield exact dimension of the distribution of Oseledets subspaces (see e.g.\ \cite{ledrappierpositive} and \cite{AvilaViana} for some notions of weak independence and corresponding results on simplicity of Lyapunov exponents).  By analogy with \cite{Feng}, one might expect that exact dimension holds for ergodic cocycles \((T,F)\) such that \(F\) takes only finitely many different values.

\medskip

The distribution of the angles between Oseledets subspaces has been studied experimentally  \cite{experimental1,experimental2}. In Hamiltonian dynamics, the classical Melnikov integral \cite{melnikov} allows to estimate these angles for perturbations of integrable systems. Arnaud \cite{marieclaude} obtained quantitative relations between the Lyapunov exponents and the distribution of the angles in the contexts of twist maps and Tonelli Hamiltonian flows. 

Another important setting consists of diffeomorphisms $f \colon M \to M$ of a compact surface $M$.
If $f$ is $C^\infty$ and topologically mixing with $h_\mathrm{top}(f)>0$, then the distribution of Oseledets angles with respect to the (necessarily unique) measure of maximal entropy satisfies a power bound, and in particular is log-integrable: this follows from 
the recent work of Buzzi, Crovisier, and Sarig \cite[Corollary~1.11 and Lemma~9.3]{BCS}. We thank Snir Ben Ovadia for this observation.


\bigskip

\noindent \textbf{Acknowledgements.} 
Our interest in the question of log-integrability of angles arose from conversations with Alexander Arbieto, Fran\c{c}ois Ledrappier, and Carlos Matheus. We are grateful for their insights. We thank Snir Ben Ovadia and Cagri Sert for valuable observations and references, and Nicolas Mart\'inez Ramos for several corrections. Finally, we thank the referee for corrections and suggestions. 


\bibliographystyle{alpha}
\bibliography{biblio}  

\end{document}